\documentclass[12pt,leqno]{article}
\usepackage[centertags]{amsmath}
\usepackage{amsfonts}
\usepackage{amssymb}
\usepackage{amsthm}
\renewcommand\baselinestretch{1.5}
\theoremstyle{plain}
\newtheorem{thm}{Theorem}[section]
\newtheorem{cor}[thm]{Corollary}
\newtheorem{lem}[thm]{Lemma}

\newtheorem{defn}[thm]{Definition}
\newtheorem{rem}[thm]{Remark}
\newtheorem{Exa}[thm]{Example}
\numberwithin{equation}{section}
\renewcommand\baselinestretch{1.5}
\begin{document}
\vspace{4cm} \vspace{4cm} \centerline{\bf On graded $S-$comultiplication modules}\vspace{2cm}
 \centerline {\baselineskip.8cm
 \centerline {  {Mohammad Hamoda$^1$ and Khaldoun Al-Zoubi$^2$  } }}
\centerline {\baselineskip.8cm {$^1$Department of Mathematics, Faculty of Applied Science, Al-Aqsa University, Gaza, Palestine.}}
\centerline {\baselineskip.8cm{P.O. Box 4051, Gaza, Palestine }}
\centerline {\baselineskip.8cm{e-mail: ma.hmodeh@alaqsa.edu.ps }}
\baselineskip.8cm {$^2$Department of Mathematics and Statistics, Faculty of Science and Arts, Jordan University of Science and Technology, Irbid, Jordan.}\\
\centerline {\baselineskip.8cm{P.O. Box 3030, Irbid-Jordan }}
\centerline {\baselineskip.8cm{e-mail: kfzoubi@just.edu.jo }}

\thispagestyle{empty}
\renewcommand\baselinestretch{1.5}
\vspace{2cm}
\begin{abstract}
In this paper, we introduce the concept of graded $S-$comultiplication modules.  Several results concerning graded $S-$comultiplication modules are proved.  We show that $N$ is a graded $S-$second submodule of a graded $S-$comultiplication $R-$module $M$ if and only if $Ann_R(N)$ is a graded $S-$prime ideal of $R$ and there exists $x\in~S$ such that $xN\subseteq\overline{x}N$ for every $\overline{x}\in~S$.
\end{abstract}
\vspace{1cm} \noindent {\bf "2010 Mathematics Subject Classification"}: 13A15, 13G05, 16W50.
\\ \noindent{\bf Keywords:}
Graded comultiplication module, Graded multiplication module, Graded second module.\\
\newpage
\section{Introduction}
Commutative algebra evolved from problems arising in number theory and algebraic geometry.  Much of the modern development of the commutative algebra emphasizes graded rings.  Once the grading is considered to be trivial, the graded theory reduces to the usual module theory.  So from this perspective, the theory of graded modules can be considered as an extension of module theory.  Graded rings play a central role in algebraic geometry and commutative algebra.  Gradings appear in many circumstances, both in elementary and advanced level.  Recently, extensive researches have been done on rings with group-graded structure, see for example \cite{a,f,p,d,c,b}.  The notion of graded multiplication modules was studied by many authors, see for example \cite{o,h,i,e}.  The notion of graded comultiplication modules which are the dual nation of graded multiplication modules was introduced and studied by Ansari-Toroghy and Farshadifar in \cite{a}.  Later, Al-Zoubi and Al-Qderat \cite{j} studied on this issue.  The objective of this paper is to construct more accurate results and concepts regarding generalizations of graded comultiplication modules.  In fact the motivation of writing this paper is two folded:\\$(i)$ To extend the concept of graded comultiplication modules to the concept of graded $S-$comultiplication modules.\\
$(ii)$ To determine when a graded module is graded $S-$comultiplication modules.  The remains of this paper is organized as follows:\\
Section $2$ concerns some basic definitions and results in the sequel of this paper.  In section $3$, the main results concerning graded $S-$comultiplication modules will be given.  Section $4$ concerns the conclusion.
\section{Preliminary Notes}
In this section we state some basic concepts and results related to graded ring theory.  We hope that this will improve the readability and understanding of this paper.
\begin{defn}\cite{k}
Let $G$ be a group with identity $e$ and $R$ be a commutative ring  with identity $1_R$.  Then, $R$ is said to be a $G-$graded ring if there exist additive subgroups $R_g$ of $R$ indexed by elements $g\in~G$ such that $R=\bigoplus_{g\in~G}R_g$ and $R_gR_h\subseteq~R_{gh}$ for all $g,~h\in~G$.  If $R_gR_h=R_{gh}$, the ring is called strongly graded ring.
\end{defn}
Consider $supp(R)=\{g\in~G:R_g\neq~0\}$.  An element $x$ of $R$ has a unique decomposition as $x=\sum_{g\in~G}x_g$ for all $g\in~G$.  Also, we write $h(R)=\bigcup_{g\in~G}R_g$.  Moreover $R_e$ is a subring of $R$ and $1_R\in~R_e$.  If an element of $R$ belongs to $h(R)$, then it is called homogeneous and any element $x_g\in~R_g$ is said to have degree $g$.
\begin{defn}\cite{k}
Let $R=\bigoplus_{g\in~G}R_g$ be a $G-$graded ring.  An ideal $I$ of $R$ is said to be a graded ideal of $R$ if $I=\bigoplus_{g\in~G}(I\bigcap~R_g)$.
\end{defn}
Clearly, $\bigoplus_{g\in~G}(I\bigcap~R_g)\subseteq~I$ and hence $I$ is a graded ideal of $R$ if $I\subseteq\bigoplus_{g\in~G}(I\bigcap~R_g)$.  Moreover $R/I$  becomes a $G-$graded ring with $g-$component $(R/I)_g=(R_g+I)/I$ for $g\in~G$.
\begin{defn}\cite{k}
Let $R$ be a $G-$graded ring and $M$ be an $R-$module.  We say that $M$ is a graded $R-$module if there exists a family of subgroups $\{M_g\}_{g\in~G}$ of $M$ such that $M=\bigoplus_{g\in~G}M_g$ (as abelian groups) and $R_gM_h\subseteq~M_{gh}$ for all $g,~h\in~G$.  If $R_gM_h=M_{gh}$, the $R-$mould $M$ is called strongly graded $R-$module.
\end{defn}
Consider $supp(M)=\{g\in~G:M_g\neq~0\}$.  Here $R_gM_h$ denotes the additive subgroups of $M$ consisting of all finite sums of elements $r_gs_h$ with $r_g\in~R_g$ and $s_h\in~M_h$.  Also, we write $h(M)=\bigcup_{g\in~G}M_g$.  If an element of $M$ belongs to $h(M)$, then it is called homogeneous and any element $x_g\in~M_g$ is said to have degree $g$.  It is clear that $M_g$ is an $R_e-$submodule of $M$ for all $g\in~G$.
\begin{defn}\cite{k}
Let $M=\bigoplus_{g\in~G}M_g$ be a $G-$graded $R-$module and $N$ be a submodule of $M$.  Then, $N$ is said to be a graded submodule of $M$ if $N=\bigoplus_{g\in~G}N_g$ where $N_g=N\bigcap~M_g$ for $g\in~G$.  In this case, $N_g$ is called the $g-$component of $N$ for $g\in~G$.  Moreover, $M/N$ becomes a $G-$graded module with $g-$component $(M/N)_g=(M_g+N)/N$ for $g\in~G$.
\end{defn}
\begin{defn}\cite{l}
Let $I$ be a graded ideal of a $G-$graded ring $R$.  Then, $I$ is said to be a graded prime ideal if $I\neq~R$; and whenever $ab\in~I$, we have $a\in~I$ or $b\in~I$, where $a,b\in~h(R)$.
\end{defn}
\begin{defn}\cite{m}
Let $R$ be a $G-$graded ring, $M$ be a graded $R-$module.  A graded submodule $N$ of $M$ is said to be a graded prime submodule of $M$ if $N\neq~M$; and whenever $r\in~h(R)$ and $m\in~h(M)$ with $rm\in~N$, then either $m\in~N$ or $r\in(N:_RM)$.
\end{defn}
\begin{defn}\cite{k}
Let $R$ be a $G-$graded ring, A nonzero graded $R-$module $M$ is said to be a graded prime module if $Ann_R(M)=Ann_R(N)$ for every nonzero graded submodule $N$ of $M$.
\end{defn}
\begin{defn}\cite{k}
Let $R$ be a $G-$graded ring.  A nonempty $S\subseteq~h(R)$ is said to be a multiplicatively closed subset of $R$ if $(i)~0\not\in~S$, $(ii)~1\in~S$, $(iii)~ab\in~S$ for all $a,b\in~S$.
\end{defn}
\begin{defn} \cite{x}
let $R$ be a $G-$graded ring, $S\subseteq~h(R)$ be a multiplicatively closed subset of $R$ and $M$ be a graded $R-$module.  A graded submodule $N$ of $M$ with $(N:_RM)\bigcap~S=\phi$ is said to be a graded $S-$prime submodule of $M$ if there exists a fixed $x\in~S$ such that whenever $rm\in~N$ for some $r\in~h(R)$ and $m\in~h(M)$, then either $xr\in(N:_RM)$ or $xm\in~N$.  In particular, a graded ideal $P$ of $R$ is said to be a graded $S-$prime if $P$ is a graded $S-$prime submodule of $M$.
\end{defn}
\begin{defn}\cite{n}
Let $R$ be a $G-$graded ring, $M$ be a graded $R-$module.  A non zero graded submodule $N$ of $M$ is said to be a graded second submodule of $M$ if $rN=0$ or $rN=N$ for every $r\in~h(R)$.
\end{defn}
\begin{defn}\cite{k}
Let $R$ be a $G-$graded ring, A graded $R-$module $M$ is said to be graded finitely generated if $M=R_{m_1}+R_{m_2}+...+R_{m_n}$ for some $m_1,m_2,...,m_n~\in~h(M)$.  $M$ is called a graded cyclic if it can be generated by a single element i.e., there exists $x\in~h(M)$ such that $M=Rx$.
\end{defn}
\begin{defn}\cite{k}
Let $R$ be a $G-$graded ring, $M,~\overline{M}$ be graded $R-$modules.  Then an $R-$homomorphism $f:M\longrightarrow\overline{M}$ is said to be a graded $R-$homomorphism if for all $m,~n\in~M$;\\
$(i)$ $f(m+n)=f(m)+f(n)$;\\
$(ii)$ $f(rm)=rf(m)$ for any $r\in~R$ and $m\in~M$;\\
$(iii)$ $f(M_g)\subseteq~\overline{M_g}$ for all $g\in~G$.
\end{defn}
\begin{defn}\cite{w}
Let $R$ be a $G-$graded ring, $M$ be a graded $R-$module and  $N$ be a graded submodule of $M$.  $M$ is said to be a graded torsion $-$ free $R-$module if whenever $r\in~h(R)$ and $m\in~M$ with $rm=0$, then either $m=0$ or $r=0$.
\end{defn}
Equivalently, $M$ is said to be a graded torsion $-$ free $R-$module if the set $T(M)=\{m\in~M:rm=~0~for~some~0\neq~r\in~h(R)\}$ is zero.  $M$ is called a graded torsion $R-$module if $T(M)=M$.
\begin{rem}\cite{s}
Let $R$ be a $G-$graded ring, $M$ be a graded $R-$module, $P$ be a graded ideal of $R$ and  $N$ be a graded submodule of $M$.  Then:\\
$(i)$ $Ann_R(M)=(0:_RM)=\{r\in~R:rM=0\}$ is a graded ideal of $R$.\\
$(ii)$ $(0:_MP)=\{m\in~M:mP=0\}$ is a graded submodule of $M$.\\
$(iii)$ $Ann_R(N)=(0:_RN)=\{r\in~R:rN=0\}$ is a graded ideal of $R$.
\end{rem}
\begin{defn}\cite{a}
Let $R$ be a $G-$graded ring.  A graded $R-$module $M$ is said to be a graded comultiplication module if for every graded submodule $N$ of $M$, there exists a graded ideal $I$ of $R$ such that $N=(0:_MI)$.
\end{defn}
\section{Results and Discussion}
We start by the following definition.
\begin{defn}
Let $R$ be a $G-$graded ring, $M$ be a graded $R-$module and $S\subseteq~h(R)$ be a multiplicatively closed subset of $R$.  $M$ is is said to be a graded $S-$comultiplication module if for each graded submodule $N$ of $M$, there exist $x\in~S$ and a graded ideal $P$ of $R$ such that $x(0:_MP)\subseteq~N\subseteq(0:_MP)$.
\end{defn}
We define the graded ring $R$ to be a graded $S-$comultiplication ring if it is a graded $S-$comultiplication module over itself.
\begin{rem}
\begin{enumerate}
\item[(i)] Every graded $R-$module $M$ with $Ann_R(M)\cap~S\neq\phi$ is a graded $S-$comultiplication module.
\item[(ii)] Every graded comultiplication  module is also graded $S-$comultiplication module.\\
The converse is true only, when $S\subseteq~U(R)$, where $U(R)$ denotes the set of all units in $R$.
\end{enumerate}
\end{rem}
In the following example, we give a graded $S-$comultiplication module that is not graded comultiplication.
\begin{Exa}
Consider $R=\mathbb{Z}$, $G=\mathbb{Z}_2$.  Define $R_0=\mathbb{Z}$ and $R_1=\{0\}$.  Then, $R$ is a $G-$graded ring.  Let $M=\mathbb{Z}=\oplus_{g\in\mathbb{Z}_2}M_g$ be a graded $\mathbb{Z}-$module where $M_0=\mathbb{Z}$ and $M_1=\{0\}$.  Let $S=\mathbb{Z}-\{0\}\subseteq~h(\mathbb{Z})$ be a multiplicatively closed subset of $R$ and consider the graded submodule $N=m\mathbb{Z}$ of $M$ where $m\neq~0,~\pm~1$.  Thus, $(0:_RAnn_R(m\mathbb{Z}))=\mathbb{Z}\neq~mZ$.  Thus, $M$ is not graded comultiplication module.  Now, we show that $M$ is a graded $S-$comultiplication module, for let $L$ be any graded submodule of $M$.  Then, $L=r\mathbb{Z}$ for some $r\in\mathbb{Z}$.  if $r=0$, then one can choose $x=1$ so that $x(0:_RAnn_R(L))=(0)=r\mathbb{Z}$.  If $r\neq0$, then one can choose $x=r$ so that $x(0:_RAnn_R(L))\subseteq~r\mathbb{Z}=L\subseteq(0:_RAnn_R(L))$.  Thus, $M$ is a graded $S-$comultiplication module
\end{Exa}
We need the following Lemma.
\begin{lem}\cite{q}
Let $R$ be a $G-$graded ring, $M$ be a graded $R-$module.  Then the following assertions hold.
\begin{enumerate}
\item[(i)] If $I$ and $J$ are graded ideals of $R$, then $I+J$ and $I\cap~J$ are graded ideals of $R$.\\
\item[(ii)] If $N$ is a graded submodule of $M$, $r\in~h(R)$, $x\in~h(M)$ and $I$ is a graded ideal of $R$, then $Rx$, $IN$ and $rN$ are graded submodules of $M$.
\item[(iii)] If $N$ and $K$ are graded submodules of $M$, then $N+K$ and $N\cap~K$ are also graded submodules of $M$ and $(N:_RM)=\{r\in~R:rM\subseteq~N\}$ is a graded ideal of $R$.
\end{enumerate}
 \end{lem}
\begin{thm}\label{1}
Let $R$ be a $G-$graded ring, $S\subseteq~h(R)$ be a multiplicatively closed subset of $R$ and $M$ be a graded $R-$module.  Then the following assertions are equivalent.
\begin{enumerate}
\item[(i)] $M$ is a graded $S-$comultiplication module.\\
\item[(ii)] For every graded submodule $N$ of $M$, there exists $x\in~S$ such that $x(0:_MAnn_R(N))\subseteq~N\subseteq(0:_MAnn_R(N))$.\\
\item[(iii)] For every graded submodules $L,~N$ of $M$ with $Ann_R(L)\subseteq~Ann_R(N)$, there exists $x\in~S$ such that $xN\subseteq~N$.
\end{enumerate}
\end{thm}
\begin{proof}
$(i)\Longrightarrow(ii)$ Assume that $M$ is a graded $S-$comultiplication module and $N$ be a graded submodule of $M$.  Then, there exist $x\in~S$ and a graded ideal $P$ of $R$ such that $x(0:_MP)\subseteq~N\subseteq(0:_MP)$.  Now, $PN=(0)$ and thus, $P\subseteq~Ann_R(N)$.  Therefore, $x(0:_MAnn_R(N))\subseteq~x(0:_MP)\subseteq~N\subseteq(0:_MAnn_R(N))$.\\
$(ii)\Longrightarrow(iii)$ Assume that $Ann_R(L)\subseteq~Ann_R(N)$ for some graded submodules $L,~N$ of $M$.  By (ii), there exist $x_1,~x_2\in~S$ such that $x_1(0:_MAnn_R(L))\subseteq~L\subseteq(0:_MAnn_R(L))$ and $x_2(0:_MAnn_R(N))\subseteq~N\subseteq(0:_MAnn_R(N))$.  Since, $Ann_R(L)\subseteq~Ann_R(N)$, we have $(0:_MAnn_R(N))\subseteq(0:_MAnn(L))$ and thus, $x_1x_2(0:_MAnn_R(N))\subseteq~x_2N\subseteq~x_2(0:_MAnn_R(N))\subseteq~x_2(0:_MAnn_R(L))\subseteq~L$.\\
$(iii)\Longrightarrow(ii)$ Assume that (ii) holds.  Let $N$ be a graded submodule of $M$.  Then it is clear that $Ann_R(N)=Ann_R(0:_MAnn_R(N))$.  Therefore, by (iii), there exists $x\in~S$ such that $x(0:_MAnn_R(N))\subseteq~N\subseteq(0:_MAnn_R(N))$.\\
$(ii)\Longrightarrow(i)$ It is clear.
\end{proof}
\begin{thm}
Let $R$ be a $G-$ strongly graded ring, $S\subseteq~h(R)$ be a multiplicatively closed subset of $R$ and $M$ be a graded $R-$module.  If $M$ is a graded $S-$comultiplication module, then $M_e$ is a graded $S-$comultiplication as an $R_e-$module.
\end{thm}
\begin{proof}
Assume that $M$ is a graded $S-$comultiplication module and let $N$ be an $R_e-$submodule of $M_e$.  Consider the graded submodule $N$ of $M$ defined by $N_g=R_gN$ for every $g\in~G$.  Then by assumption, we have $x(0:_MP)\subseteq~N\subseteq(0:_MP)$ for some $x\in~S$ and graded ideal $P$ of $R$.  Since $R$ is strongly graded, then $P=RP_e$ by \cite{r}.  Hence, one can easily see that $(0:_MP)=(0:_MP_e)$.  Thus, $x(0:_{M_e}P_e)\subseteq~N_e\subseteq(0:_{M_e}P_e)$.  Therefore, $M_e$ is a graded $S-$comultiplication an $R_e-$module.
\end{proof}
\begin{defn}
Let $R$ be a $G-$ graded ring, $S\subseteq~h(R)$ be a multiplicatively closed subset of $R$ and $M$ be a graded $R-$module.  Then, $S^{-1}M$ is a graded $S^{-1}R-$ module where,\\The ring of fraction is defined by:
$$(S^{-1}R)_g=\{\frac{r}{x}:r\in~h(R),~x\in~S~and~g=deg~r-deg~x~for~all~g\in~G\}$$
The quotient module $M$ is thus defined by:
$$(S^{-1}M)_g=\{\frac{m}{x}:m\in~h(M),~x\in~S~and~g=deg~m-deg~x~for~all~g\in~G\}$$
The saturation $S^\star$ of $S$ is defined by:
$$S^\star=\{x\in~h(R):x~divides~s~for~some~s\in~S\}\subseteq~h(R)$$ is a multiplicatively closed subset of $R$ containing $S$.
\end{defn}
\begin{thm}
Let $R$ be a $G-$ graded ring, $S\subseteq~h(R)$ be a multiplicatively closed subset of $R$ and $M$ be a graded $R-$module.  Then the following assertions hold.  \begin{enumerate}
\item[(i)] Let $S_1\subseteq~h(R)$ and $S_2\subseteq~h(R)$ be two multiplicatively closed subsets of $R$ such that $S_1\subseteq~S_2$.  If $M$ is a graded $S_1-$comultiplication module, then $M$ is also a graded $S_2-$comultiplication module.\\
\item[(ii)] $M$ is a graded $S-$comultiplication module if and only if $M$ is a graded $S^\star-$comultiplication module .
\end{enumerate}
\end{thm}
\begin{proof}
\begin{enumerate}
\item[(i)] It is clear.\\
\item[(ii)] Assume that $M$ is a graded $S-$comultiplication module.  Since $S\subseteq~S^\star$, then the result follows from part (i).\\
  Conversely; assume that $M$ is a graded $S^\star-$comultiplication module, where $S^\star$ is the saturation of $S$.  Let $N$ be a graded submodule of $M$.  Since $M$ is a graded $S^\star-$comultiplication module, then there exists $x\in~S^\star$ such that $x(0:_MAnn_R(N))\subseteq~N\subseteq(0:_MAnn_R(N))$ by Theorem \ref{1}.  Now, $x\in~S^\star$ implies that there exists $s\in~S$ such that $x$ divides $s$, that is there exists $r\in~h(R)$ such that $s=rx$.  Thus, $s(0:_MAnn_R(N))\subseteq~x(0:_MAnn_R(N))\subseteq~N\subseteq(0:_MAnn_R(N))$.  Therefore, $M$ is graded $S-$comultiplication module.
\end{enumerate}
\end{proof}
Now, we introduce the following definition.
\begin{defn}
Let $R$ be a $G-$graded ring, $S\subseteq~h(R)$ be a multiplicatively closed subset of $R$ and $M$ be a graded $R-$module.  A graded submodule $N$ of $M$ is said to be a graded $S-$finite submodule if there exists a finitely generated graded submodule $L$ of $M$ such that $xN\subseteq~L\subseteq~N$ for some $x\in~S$.  Also, $M$ is said to be a graded $S-$Noetherian module if each graded submodule is graded $S-$finite.  In particular, $R$ is said to be a graded $S-$Noetherian ring if it is a graded $S-$Noetherian $R-$module.
\end{defn}
\begin{thm}
Let $R$ be a $G-$graded $S-$Noetherian ring and $M$ be a graded $S-$comultiplication module.  Then, $S^{-1}M$ is a graded comultiplication module.
\end{thm}
\begin{proof}
Let $K$ be a graded submodule of $S^{-1}M$.  Then, $K=S^{-1}N$ for some graded submodule $N$ of $M$.  Since $M$ is a graded $S-$comultiplication module, then there exists $x\in~S$ such that $x(0:_MP)\subseteq~N\subseteq(0:_MP)$ for some graded ideal $P$ of $R$.  Thus, we have $S^{-1}(x(0:_MP))=S^{-1}((0:_MP))\subseteq~S^{-1}N\subseteq~S^{-1}((0:_MP))$, that is $S^{-1}N=S^{-1}((0:_MP))$.  We need to show that $S^{-1}((0:_MP))=(0:_{S^{-1}M}S^{-1}P)$.  Let $\frac{m}{y}\in~S^{-1}((0:_MP))$, where $m\in(0:_MP)$ and $y\in~S$.  Then, we have $Pm=(0)$ and so $(S^{-1}P)(\frac{m}{y})=(0)$.  This implies that $\frac{m}{y}\in(0:_{S^{-1}M}S^{-1}P)$ and thus, $S^{-1}((0:_MP))\subseteq(0:_{S^{-1}M}S^{-1}P)$.  Now, let $\frac{m}{y}\in(0:_{S^{-1}M}S^{-1}P)$.  Then, $(S^{-1}P)(\frac{m}{y})=(0)$.  This implies that for each $a\in~P$, there exists $z\in~S$ such that $zam=0$.  Since $R$ is a graded $S-$Noetherian ring, then $P$ is graded $S-$finite.  Thus, there exists $t\in~S$ and $a_1,a_2,...,a_n\in~P\cap~h(R)$ such that $tP\subseteq\{a_1,a_2,...,a_n\}\subseteq~P$.  As $(S^{-1}P)(\frac{m}{y})=(0)$ and $a_i\in~P$ $\forall~i\in\{1,...,n\}$, then there exists $x_i\in~S$ such that $x_ia_im=0$.  Now, put $r=x_1x_2...x_nt\in~S$.  Then, we have $ra_im=0$ for all $a_i$ and thus, $rPm=0$.  Then, we deduce $\frac{m}{y}=\frac{rm}{ry}=S^{-1}((0:_MP))$ and thus, $(0:_{S^{-1}M}S^{-1}P)\subseteq~S^{-1}((0:_MP))$.  Thus, $S^{-1}((0:_MP))=(0:_{S^{-1}M}S^{-1}P)$ and so $K=S^{-1}N=(0:_{S^{-1}M}S^{-1}P)$.  Therefore, $S^{-1}M$ is a graded comultiplication module.
\end{proof}
Let $R$ be a $G-$graded ring and $S\subseteq~h(R)$ be a multiplicatively closed subset of $R$.  $S$ is said to satisfy the maximal multiple condition if there exists $x\in~S$ such that $t$ divides $x$ for each $t\in~S$.
\begin{thm}
Let $R$ be a $G-$graded ring, $S\subseteq~h(R)$ be a multiplicatively closed subset of $R$ satisfying the maximal multiple condition and $M$ be a graded $R-$module.  Then $M$ is a graded $S-$comultiplication module if and only if $S^{-1}M$ is a graded comultiplication module.
\end{thm}
\begin{proof}
Assume that $K$ is a graded submodule of $S^{-1}M$.  Then $K=S^{-1}N$ for some graded submodule $N$ of $M$.  Since $M$ is a graded $S-$comultiplication module, then there exist $x\in~S$ and a graded ideal $P$ of $R$ such that $x(0:_MP)\subseteq~N\subseteq(0:_MP)$.  Thus, $PN=(0)$ and so $S^{-1}(PN)=(S^{-1}P)(S^{-1}N)=(0)$. Thus, we have $S^{-1}N\subseteq(0:_{S^{-1}M}S^{-1}P)$.  Let $\frac{m}{s}\in(0:_{S^{-1}M}S^{-1}P)$.  Then, we get $\frac{a}{1}\frac{m}{s}=0$ for each $a\in~P$ and thus, $yam=0$ for some $y\in~S$.  As $S$ satisfies the maximal multiple condition, then there exists $z\in~S$ such that $y$ divides $z$ for each $y\in~S$.  This implies that $z=yr$ for some $r\in~h(R)$.  Then, we have $zam=ryam=0$.  Then, we have $Pzm=0$ and so $zm\in(0:_MP)$.  Thus, $xzm\in~x(0:_MP)\subseteq~N$ and so $\frac{m}{s}=\frac{xzm}{xzs}\in~S^{-1}N$.  Thus, we have $S^{-1}N=(0:_{S^{-1}M}S^{-1}P)$ and hence $S^{-1}M$ is a graded comultiplication module.\\
Conversely; assume that $S^{-1}M$ is a graded comultiplication module and let $N$ be a graded submodule of $M$.  Since $S^{-1}M$ is a graded comultiplication module, then $S^{-1}N=(0:_{S^{-1}M}S^{-1}P)$ for some graded ideal $P$ of $R$.  Then, we have $(S^{-1}P)(S^{-1}N)=S^{-1}(PN)=0$.  Then for each $r\in~P,~m\in~N$, we have $\frac{rm}{1}=0$ and thus, $xrm=0$ for some $x\in~S$.  Since $S\subseteq~h(R)$ is a multiplicatively closed subset of $R$ satisfying the maximal multiple condition, then there exists $y\in~S$ such that $yrm=0$ and so $yPN=0$.  Thus, $N\subseteq(0:_MyP)$.  Now, let $m\in(0:_MyP)$.  Then, $Pym=0$, so it is easily seen that $(S^{-1}P)(\frac{m}{1})=0$.  Thus, we have $\frac{m}{1}\in(0:_{S^{-1}M}S^{-1}P)=S^{-1}N$.  Then, there exists $z\in~S$ such that $zm\in~N$.  Again by the maximal multiple condition, $ym\in~N$.  Thus, we have $y(0:_MyP)\subseteq~N\subseteq(0:_MyP)$.  Since $yP$ is a graded ideal of $R$, then $M$ is a graded $S-$comultiplication module.
\end{proof}
\begin{thm}\label{2}
Let $R$ be a $G-$graded ring, $S\subseteq~h(R)$ be a multiplicatively closed subset of $R$, $M$ and $\overline{M}$ be graded $R-$modules, and $f:M\longrightarrow\overline{M}$ be a graded $R-$homomorphisim with $a~Ker(f)=0$ for some $a\in~S$.  Then the following assertions hold.
\begin{enumerate}
\item[(i)] If $\overline{M}$ is a graded $S-$comultiplication module, then $M$ is a graded $S-$comultiplication module.\\
\item[(ii)] If $f$ is a graded $R-$epimorphism and $M$ is a graded $S-$comultiplication module, then $\overline{M}$ is a graded $S-$comultiplication module.
\end{enumerate}
\end{thm}
\begin{proof}
\begin{enumerate}
\item[(i)] Let $N$ be a graded submodule of $M$.  Since $\overline{M}$ is a graded $S-$comultiplication module, then there exist $x\in~S$ and a graded ideal $P$ of $R$ such that $x(0:_{\overline{M}}P)\subseteq~f(N)\subseteq(0:_{\overline{M}}P)$.  Thus, we have $Pf(N)=f(PN)=0$ and so $PN\subseteq~Ker(f)$.  Since $a~Ker(f)=0$, we have $aPN=(0)$ and so $N\subseteq(0:_MaP)$.  Now, we will show that $a^2x(0:_MaP)\subseteq~N\subseteq(0:_MaP)$.  Let $m\in(0:_MaP)$.  Then we have $aPm=0$ and so $f(aPm)=aPf(m)=Pf(am)=0$.  This implies that $f(am)\in(0:_{\overline{M}}P)$.  Thus, we have $xf(am)=f(xam)\in~x(0:_{\overline{M}}P)\subseteq~f(N)$ and so there exists $y\in~N$ such that $f(xam)=f(y)$ and so $xam-y\in~Ker(f)$.  Thus, we have $a(xam-y)=0$ and so $a^2xm=ay$.  Then we obtain $a^2x(0:_MaP)\subseteq~aN\subseteq~N\subseteq(0:_MaP)$.  Now, put $a^2x=s\in~S$ and $J=aP$.  Thus, $s(0:_MJ)\subseteq~N\subseteq(0:_MJ)$.  Therefore, $M$ is a graded $S-$comultiplication module.\\
\item[(ii)] Let $\overline{N}$ be a graded submodule of $\overline{M}$.  Since $M$ is a graded $S-$comultiplication module, then there exist $x\in~S$ and a graded ideal $P$ of $R$ such that $x(0:_MP)\subseteq~f^{-1}(\overline{N})\subseteq(0:_MP)$.  This implies that $Pf^{-1}(\overline{N})=(0)$ and so $f(Pf^{-1}(\overline{N}))=P\overline{N}=(0)$ since $f$ is graded surjective.  Then, we have $\overline{N}\subseteq(0:_{\overline{M}}P)$.  On the other hand, we get $f(x(0:_MP))=xf((0:_MP))\subseteq~f(f^{-1}(\overline{N}))=\overline{N}$.  Now, let $\overline{m}\in(0:_{\overline{M}}P)$.  Then, $P\overline{m}=0$.  Since $f$ is graded epimorphism, then there exists $m\in~M$ such that $\overline{m}=f(m)$.  Then we have $P\overline{m}=Pf(m)=f(Pm)=0$ and so $Pm\subseteq~Ker(f)$.  Since $a~Ker(f)=0$, we have $aPm=(0)$ and so $am\in(0:_MP)$.  Then, we get $f(am)=af(m)=a\overline{m}\in~f((0:_MP))$.  Thus, we have $a(0:_{\overline{M}}P)\subseteq~f((0:_MP))$ and hence $xa(0:_{\overline{M}}P)\subseteq~xf((0:_MP))\subseteq~\overline{N}\subseteq(0:_{\overline{M}}P)$.  Therefore, $\overline{M}$ is a graded $S-$comultiplication module.
\end{enumerate}
\end{proof}
\begin{cor}
Let $R$ be a $G-$ graded ring, $S\subseteq~h(R)$ be a multiplicatively closed subset of $R$, $M$ be a graded $R-$module and $N$ be a graded submodule of $M$. Then the following assertions hold.
\begin{enumerate}
\item[(i)] If $M$ is a graded $S-$comultiplication module, then $N$ is a graded $S-$comultiplication module.\\
\item[(ii)] If $M$ is a graded $S-$comultiplication module and $aM\subseteq~N$ for some $a\in~S$, then $M/N$ is a graded $S-$comultiplication $R-$ module.
\end{enumerate}
\end{cor}
\begin{proof}
Follows directly from Theorem \ref{2}
\end{proof}
Let $R_1$ and $R_2$ be $G-$graded rings. As in \cite{t}, $R=R_1\times~R_2$ is a $G-$graded ring with $R_g=(R_1)_g\times~(R_2)_g$ for all $g\in~G$.  Let $M_1$ be a $G-$graded $R_1-$module, $M_2$ be a $G-$graded $R_2-$module and $R=R_1\times~R_2$.  Then $M=M_1\times~M_2$ is a $G-$graded $R-$module with $M_g=(M_1)_g\times~(M_2)_g$ for all $g\in~G$.  Also, if $S_1\subseteq~h(R_1)$ is a multiplicatively closed subset of $R_1$ and $S_2\subseteq~h(R_2)$ is a multiplicatively closed subset of $R_2$, then $S=S_1\times~S_2$ is a multiplicatively closed subset of $R$.  Furthermore, each graded submodule of $M$ is of the form $N=N_1\times~N_2$, where $N_i$ is a graded submodule of $M_i$ for $i=1,~2$.
\begin{thm}
Let $R_i$ be a $G-$graded ring, $M_i$ be a graded $R_i$ module and $S_i\subseteq~h(R_i)$ be a multiplicatively closed subset of $R_i$ for each $i\in\{1,~2\}$. Suppose that $M=M_1\times~M_2$ be a graded $R=R_1\times~R_2-$module and $S=S_1\times~S_2$ be a multiplicatively closed subset of $R$.  If $M$ is a graded $S-$comultiplication $R-$module, then $M_1$ is a graded $S_1-$comultiplication $R_1-$modue and $M_2$ is a graded $S_2-$comultiplication $R_2-$modue.
\end{thm}
\begin{proof}
Assume that $M$ is a graded $S-$comultiplication $R-$module.  Let $K_1$ be a graded submodule of $M_1$.  Then, $K_1\times\{0\}$ is a graded submodule of $M$.  Since $M$ is a graded $S-$comultiplication $R-$module, then there exist $x=(x_1,~x_2)\in~S_1\times~S_2$ and a graded ideal $J=P_1\times~P_2$ of $R$ such that $(x_1,~x_2)(0:_MP_1\times~P_2)\subseteq~K_1\times\{0\}\subseteq(0:_MP_1\times~P_2)$, where $P_i$ is a graded ideal of $R_i$.  Then, we can easily get $x_1(0:_{M_1}P_1)\subseteq~K_1\subseteq(0:_{M_1}P_1)$.  Therefore, $M_1$ is a graded $S_1-$comultiplication $R_1-$modue.  Similarly, taking a graded submodule $K_2$ of $M_2$ and a graded submodule $\{0\}\times~K_2$ of $M$, we can show that $M_2$ is a graded $S_2-$comultiplication $R_2-$modue.
\end{proof}
\begin{thm}
Let $R_i$ be a $G-$graded ring, $M_i$ be a graded $R_i$ module and $S_i\subseteq~h(R_i)$ be a multiplicatively closed subset of $R_i$ for each $i\in\{1,2,...,n\}$. Suppose that $M=M_1\times~M_2\times...\times~M_n$ be a graded $R=R_1\times~R_2\times...\times~R_n-$module and $S=S_1\times~S_2\times...\times~S_n$ be a multiplicatively closed subset of $R$.  If $M$ is a graded $S-$comultiplication $R-$module, then $M_i$ is a graded $S_i-$comultiplication $R_i-$modue for each $i\in\{1,2,...,n\}$.
\end{thm}
\begin{proof}
Use induction on $n$.
\end{proof}
Now we give the following definition.
\begin{defn}
Let $R$ be a $G-$graded ring, $S\subseteq~h(R)$ be a multiplicatively closed subset of $R$ and $M$ be a graded $R-$module.  $M$ is is said to be a graded $S-$cyclic module if there exists $x\in~S$ such that $xM\subseteq~Rm\subseteq~M$ for some $m\in~h(M)$.
\end{defn}
\begin{thm}
Let $R$ be a $G-$graded ring, $S\subseteq~h(R)$ be a multiplicatively closed subset of $R$, $M$ be a graded $S-$comultiplication $R-$module and $N$ be a minimal graded ideal of $R$ such that $(0:_MN)=0$.  Then, $M$ is a graded $S-$cyclic module.
\end{thm}
\begin{proof}
Chose $0\neq~m\in~h(M)$.  Since $M$ is a graded $S-$comultiplication $R-$module, then there exist $x\in~S$ and a graded ideal $P$ of $R$ such that $x(0:_MP)\subseteq~Rm\subseteq(0:_MP)$.  Since $(0:_MN)=0$, we have $x((0:_MN):_MP)\subseteq~Rm\subseteq((0:_MN):_MP)$.  Then, $x(0:_MNP)\subseteq~Rm\subseteq(0:_MNP)$.  Since $0\subseteq~NP\subseteq~N$ and $N$ is minimal graded ideal of $R$, then either $NP=N$ or $NP=0$.  Case (i): $NP=N$, then $x(0:_MN)\subseteq~Rm\subseteq(0:_MN)$.  This means that $Rm=0$, a contradiction.  Case (ii): $NP=0$, then $x(0:_M0)\subseteq~Rm\subseteq(0:_M0)$.  This means that $xM\subseteq~Rm\subseteq~M$ and hence $M$ is a graded $S-$cyclic.
\end{proof}
\begin{thm}
Let $R$ be a $G-$graded ring, $S\subseteq~h(R)$ be a multiplicatively closed subset of $R$, $M$ be a graded $S-$comultiplication $R-$module and $\{M_i\}_{i\in~I}$ be a collection of graded submodules of $M$ with $\bigcap_iM_i=0$.  Then, for every graded submodule $K$ of $M$, there exists an $x\in~S$ such that $x\bigcap_i(K+M_i)\subseteq~K\subseteq\bigcap_i(K+M_i)$.
\end{thm}
\begin{proof}
Let $K$ be a graded submodule of $M$.  Since $M$ is a graded $S-$comultiplication module, then $x(0:_MAnn_R(K))\subseteq~K\subseteq(0:_MAnn_R(K))$ for some $x\in~S$.  Thus, $x(\bigcap_iM_i:_MAnn_R(K))\subseteq~K\subseteq(\bigcap_iM_i:_MAnn_R(K))$ since $\bigcap_iM_i=0$.  Thus, $x\bigcap_i(M_i:_MAnn_R(K))\subseteq~K\subseteq\bigcap_i(M_i:_MAnn_R(K))$.  Therefore, $x\bigcap_i(K+M_i)\subseteq~x\bigcap_i(M_i:_MAnn_R(K))\subseteq~K\subseteq\bigcap_i(k+M_i)$
\end{proof}
\begin{thm}
Let $R$ be a $G-$graded ring, $S\subseteq~h(R)$ be a multiplicatively closed subset of $R$ and $M$ be a graded $S-$comultiplication $R-$module.  Then for each graded submodule $N$ of $M$ and each graded ideal $P$ of $R$ with $N\subseteq~x(0:_MP)$ for some $x\in~S$, there exists a graded ideal $J$ of $R$ with $P\subseteq~J$ and $x(0:_MJ)\subseteq~N$
\end{thm}
\begin{proof}
Assume that $N$ is a graded submodule of $M$.  Since, $M$ is a graded $S-$comultiplication module, then we have $x(0:_MAnn_R(N)\subseteq~N\subseteq(0:_MAnn_R(N)$ for some $x\in~S$.  Thus, we have $x(0:_MAnn_R(N)\subseteq~N\subseteq~x(0:_MP)$.  Since $P\subseteq~P+Ann_R(N)$, one can take $J$ to be $P+Ann_R(N)$.  Therefore, $x(0:_MJ)=x(0:_MP+Ann_R(N))\subseteq~x(0:_MP)\bigcap~x(0:_MAnn_R(N))\subseteq~x(0:_MAnn_R(N))\subseteq~N$.
\end{proof}
\begin{thm}
Every graded $S-$comultiplication module is either graded $S-$cyclic or graded torsion.
\end{thm}
\begin{proof}
Let $R$ be a $G-$graded ring, $S\subseteq~h(R)$ be a multiplicatively closed subset of $R$ and $M$ be a graded $S-$comultiplication $R-$module.  Suppose that $M$ is not graded $S-$cyclic and $Ann_R(m)=0$ for some $m\in~M$.  As $Rm$ is a graded submodule of $M$ and $M$ a graded $S-$comultiplication $R-$module, we get $x(0:_MAnn_R(m))\subseteq~Rm\subseteq(0:_MAnn_R(m))$ for some $x\in~S$.  Thus, $xM\subseteq~Rm\subseteq~M$, a contradiction.  Thus, $Ann_R(m)\neq~0~\forall~m\in~M$ and hence $M$ is graded torsion.
\end{proof}
Now, we introduce the following definition.
\begin{defn}
Let $R$ be a $G-$graded ring, $S\subseteq~h(R)$ be a multiplicatively closed subset of $R$, $M$ be a graded $R-$module and $N$ be a nonzero graded submodule of $M$module.  $M$ is is said to be a graded $S-$minimal submodule if $L\subseteq~N$ for some graded submodule $L$ of $M$, then there exists $x\in~S$ with $xN\subseteq~L$.
\end{defn}
\begin{thm}
Every graded $S-$comultiplication prime module is graded $S-$minimal.
\end{thm}
\begin{proof}
Let $R$ be a $G-$graded ring, $S\subseteq~h(R)$ be a multiplicatively closed subset of $R$ and $M$ be a graded $S-$comultiplication prime $R-$module.  Let $L$ be a nonzero graded submodule of $M$.  Since $M$ is graded prime, we have $Ann_R(L)=Ann_R(M)$.  Also, $(0:_MAnn_R(L))=(0:_MAnn_R(M))$.  Since, $M$ is graded $S-$comultiplication module, then $x(0:_MAnn_R(L))\subseteq~L\subseteq(0:_MAnn_R(L))$ for some $x\in~S$.  Thus, we have $x(0:_MAnn_R(M))\subseteq~L\subseteq(0:_MAnn_R(M))$ and hence $xM\subseteq~L\subseteq~M$.  Therefore, $M$ is graded $S-$minimal.
\end{proof}
Now, we need the following definition.
\begin{defn}
Let $R$ be a $G-$graded ring, $S\subseteq~h(R)$ be a multiplicatively closed subset of $R$, $M$ and $\overline{M}$ be two graded $R-$modules and $f:M\longrightarrow\overline{M}$ be a graded $R-$homomorphism.\\
\begin{enumerate}
\item[(i)] If there exists $x\in~S$ such that $f(m)=0$, where $m\in~h(M)$ implies that $xm=0$, then $f$ is called a graded $S-$injective.\\
\item[(ii)] If there exists $x\in~S$ such that $x\overline{M}\subseteq~Im(f)$, then $f$ is called a graded $S-$surjective.\\
\item[(iii)] $f$ is called a graded $S-$zero if there exists $x\in~S$ such that $xf(m)=0~\forall~m\in~h(M)$.
\end{enumerate}
\end{defn}
\begin{thm}\label{3}
Let $R$ be a $G-$graded ring, $S\subseteq~h(R)$ be a multiplicatively closed subset of $R$, $M$ be a graded $R-$module and $N$ be a graded submodule of $M$ with $(N:_RM)\bigcap~S=\phi$.  Then the following assertions are equivalent.\\
\begin{enumerate}
\item[(i)] $N$ is a graded $S-$prime submodule of $M$.\\
\item[(ii)] There exists a fixed $x\in~S$ such that for any $r\in~h(R)$, the homothety $M/N~\underrightarrow{r}~M/N$ ($End(M)$ given by multiplication of $r$) is either graded $S-$zero or graded $S-$injective with respect to $x\in~S$.
\end{enumerate}
\end{thm}
\begin{proof}
$(i)\Longrightarrow(ii)$ Assume that $N$ is a graded $S-$prime submodule of $M$.  Then, there exists a fixed $x\in~S$ such that $rm\in~N$ for some $r\in~h(R)$, $m\in~h(M)$ implies that $xrM\subseteq~N$ or $xm\subseteq~N$.  Now, take $r\in~R$ and assume that the homothety $M/N~\underrightarrow{r}~M/N$ is not graded $S-$injective with respect to $x\in~S$.  Thus, there exists $m\in~h(M)$ with $r(m+N)=rm+N=0_{M/N}$, but $x(m+N)\neq~0_{M/N}$.  Thus, $rm\in~N$ and $xm\not\in~N$.  Now, since $N$ is a graded $S-$prime submodule, then $xr\in(N:_RM)$ and thus, $xrt\in~N$ for some $t\in~h(M)$.  Thus, we have $xr(t+N)=0_{M/N}$ for each $t\in~h(M)$, that is, the homothety $M/N~\underrightarrow{r}~M/N$ is graded $S-$zero with respect to $x\in~S$.\\
$(ii)\Longrightarrow(i)$ Assume that $(ii)$ holds, let $rm\in~N$ for some $r\in~h(R)$ and $m\in~h(M)$.  Suppose that $xm\not\in~N$.  Then, $M/N~\underrightarrow{r}~M/N$ is not graded $S-$injectiv.  Thus, by $(ii)$, $M/N~\underrightarrow{r}~M/N$ is graded $S-$zero with respect to $x\in~S$, namely, $xr(t+N)=0_{M/N}$ for each $t\in~h(M)$.  Thus, $xr\in(N:_RM)$.  Therefore, $N$ is a graded $S-$prime submodule of $M$.
\end{proof}
\begin{rem}
Take $S\subseteq~U(R)$ in Theorem \ref{3}, one can easily see that a graded submodule $N$ of $M$ is a graded prime submodule if and only if every homothety $M/N~\underrightarrow{r}~M/N$ is either graded injective or graded zero.
\end{rem}
Now, we need the following definition.
\begin{defn}
Let $R$ be a $G-$graded ring, $S\subseteq~h(R)$ be a multiplicatively closed subset of $R$ and $M$ be a graded $R-$module.  A nonzero graded submodule $N$ of $M$ with $Ann_R(N)\bigcap~S=\phi$ is said to be a graded $S-$second submodule if there exists $x\in~S$ with $xrN=0$ or $xrN=xN$ for each $r\in~h(R)$.
\end{defn}
\begin{thm}\label{4}
Let $R$ be a $G-$graded ring, $S\subseteq~h(R)$ be a multiplicatively closed subset of $R$, $M$ be a graded $R-$module and $N$ be a graded submodule of $M$ with $Ann_R(N)\bigcap~S=\phi$.  Then the following assertions are equivalent.\\
\begin{enumerate}
\item[(i)] $N$ is a graded $S-$second submodule of $M$.\\
\item[(ii)] There exists $x\in~S$ such that for each $r\in~h(R)$, the homothety $M/N~\underrightarrow{r}~M/N$ ($End(M)$ given by multiplication of $r$) is either graded $S-$zero or graded $S-$surjective with respect to $x\in~S$.\\
\item[(iii)] There exists a fixed $x\in~S$ so that for each $x\in~h(R)$, either $xrN=0$ or $xN=rN$.
\end{enumerate}
\end{thm}
\begin{proof}
Proceed similar to Theorem \ref{3}
\end{proof}
Recall that a graded submodule $N$ of a graded $R-$module $M$ is said to be completely irreducible if $N$ is not the intersection of any graded submodules of $M$ that properly contain it.
\begin{thm}\label{5}
Let $R$ be a $G-$graded ring, $S\subseteq~h(R)$ be a multiplicatively closed subset of $R$, $M$ be a graded $S-$comultiplication $R-$module and $N$ be a graded submodule of $M$ with $Ann_R(N)\bigcap~S=\phi$.  Then the following assertions are equivalent.\\
\begin{enumerate}
\item[(i)] $N$ is a graded $S-$second submodule of $M$.\\
\item[(ii)] $Ann_R(N)$ is a graded $S-$prime ideal of $R$ and there exists $x\in~S$ such that $xN\subseteq\overline{x}N$ for every $\overline{x}\in~S$
\end{enumerate}
\end{thm}
\begin{proof}
$(i)\Longrightarrow(ii)$ Assume that $N$ is a graded $S-$second submodule of $M$.  Let $r\overline{r}\in~Ann_R(N)$ for some $r,~\overline{r}\in~h(R)$.  Since $N$ is a graded $S-$second submodule of $M$, then there exists $x\in~S$ such that $rxN=xN$ or $rxN=0$ and $\overline{r}xN=xN$ or $\overline{r}xN=0$.  If $rxN=0$ or $\overline{r}xN=0$, then $Ann_R(N)$ is a graded $S-$prime ideal of $R$.  If $rxN=xN$, then $0=\overline{r}rxN=\overline{r}xN$, a contradiction.  If $\overline{r}xN=xN$, then $0=r\overline{r}xN=rxN$, a contradiction.  Thus, in any case, $rxN=0$ or $\overline{r}xN=0$, and therefore, $Ann_R(N)$ is a graded $S-$prime ideal of $R$.  Again, $N$ is a graded $S-$second submodule of $M$ implies that there exists $x\in~S$ such that $rN\subseteq~L$ for each $r\in~h(R)$ and a submodule $L$ of $M$.  Thus, $xN\subseteq~L$ or $xrN=0$.  Let $K$ be a completely irreducible graded submodule of $M$ with $\overline{x}N\subseteq~K$.  Then, $xN\subseteq~K$ or $\overline{x}xN=0$.  Since $Ann_R(N)\bigcap~S=\phi$, we have $xN\subseteq~K$.  Therefore, $xN\subseteq\overline{x}N$.\\
 $(ii)\Longrightarrow(i)$ Assume that $Ann_R(N)$ is a graded $S-$prime ideal of $R$.  We need to show that $N$ is a graded $S-$second submodule of $M$.  Let $r\in~h(R)$.  Since $Ann_R(N)$ is a graded $S-$prime ideal of $R$, by (\cite{x}, Lemma 4.2, Proposition 4.3), there exists $x\in~S$ such that $Ann_R(xN)$ is a graded prime ideal and $Ann_R(\overline{x}N)\subseteq~Ann_R(xN)$ for every $\overline{x}\in~S$.  Assume that $xrN\neq(0)$.  Now, we need to show that $xN\subseteq~rN$.  Since $M$ is a graded $S-$comultiplication module, then there exist $\overline{x}\in~S$ and a graded ideal $P$ of $R$ with $\overline{x}(0:_MP)\subseteq~rN\subseteq(0:_MP)$.  Thus, $rP\subseteq~Ann_R(N)$.  Since $Ann_R(N)$ is a graded $S-$prime ideal of $R$, then there exists $x\in~S$ such that $xr\in~Ann_R(N)$ or $xP\subseteq~Ann_R(N)$ by (\cite{x}, Lemma 4.2, Proposition 4.3).  The first case is impossible since $xrN\neq(0)$.  Thus, we have $P\subseteq~Ann_R(xN)$.  Thus, we have $\overline{x}x(0:_MAnn_R(xN)\subseteq\overline{x}(0:_MP)\subseteq~rN$.  Thus, $\overline{x}x^2N\subseteq\overline{x}x(0:_MAnn_R(xN)\subseteq~xN$.  Then, by $(ii)$, $xN\subseteq\overline{x}x^2N\subseteq~rN$.  Therefore, by Theorem \ref{4}, $N$ is a graded $S-$second submodule of $M$.
\end{proof}
\begin{thm}
Let $R$ be a $G-$graded ring, $S\subseteq~h(R)$ be a multiplicatively closed subset of $R$, $M$ be a graded comultiplication $R-$module and $N$ be a graded submodule of $M$ with $Ann_R(N)\bigcap~S=\phi$.  Then the following assertions are equivalent.\\
\begin{enumerate}
\item[(i)] $N$ is a graded second submodule of $M$.\\
\item[(ii)] $Ann_R(N)$ is a graded prime ideal of $R$.
\end{enumerate}
\end{thm}
\begin{proof}
If we take $S\subseteq~U(R)$, then the concepts of graded $S-$comultiplication modules and graded comultiplication modules are the same.  On the other hand, the concepts of graded second submodules and graded $S-$second submodules are the same.  Therefore, the rest follows from Theorem \ref{5}
\end{proof}
\begin{thm}
Let $R$ be a $G-$graded ring, $S\subseteq~h(R)$ be a multiplicatively closed subset of $R$, $M$ be a graded $S-$comultiplication $R-$module and $N$ be a graded $S-$second submodule of $M$.  If $N\subseteq\sum_{i=1}^nN_i$ for some graded submodules $N_1,N_2,...,N_n$ of $M$, then there exists $x\in~S$ such that $xN\subseteq~N_i$ for some $i\in\{1,2,...,n\}$.
\end{thm}
\begin{proof}
Assume that $N$ is a graded $S-$second submodule of a graded $S-$comultiplication module $M$ such that $N\subseteq\sum_{i=1}^nN_i$ for some graded submodules $N_1,N_2,...,N_n$ of $M$.  Then, we have $Ann_R(\sum_{i=1}^nN_i)=\bigcap_{i=1}^nAnn_R(N_i)\subseteq~Ann_R(N)$.  Since $N$ is a graded $S-$second submodule of $M$, we have $Ann_R(N)$ is a graded $S-$prime ideal of $R$ by Theorem \ref{5}.  Then, by (\cite{x}, Corollary 2.5), there exists $x\in~S$ such that $x~Ann_R(N_i)\subseteq~Ann_R(N)$ for some $i\in\{1,2,...,n\}$.  Thus, $Ann_R(N_i)\subseteq~Ann_R(xN)$.  Then by Theorem \ref{1} $(iii)$, $xmN\subseteq~N_i$ for some $m\in~S$.  Therefore, we are done.
\end{proof}

\section{Conclusion}
Here, we represented a new form of the graded theory.  We discussed and proved new theorems in this area.  We investigated the relations between graded $S-$comultiplication modules and graded $S-$cyclic modules. Also, we dedicated the study to graded $S-$second modules of graded $S-$comultiplication modules.  We can generalize the notion of graded $S-$comultiplication modules to the notion of graded $S-$comultiplication $2-$absorbing modules in the next work.\\
\textbf{Acknowledgement}\\
The authors are grateful to the anonymous referee for his/her helpful comments and suggestions aimed at improving this paper.

\end{document}